\documentclass{amsart}
\usepackage[sorted-cites]{amsrefs}
\usepackage{graphicx}
\usepackage{latexsym}
\usepackage{amsthm}

\usepackage{amsfonts}
\usepackage{xcolor}
\usepackage[all]{xy}
\usepackage{dsfont}

\usepackage{amssymb, mathrsfs, amsfonts, amsmath}
\usepackage{amsbsy}
\usepackage{amsfonts}

\newtheorem{corollary}{Corollary}

\newtheorem{definition}{Definition}
\newtheorem{lemma}{Lemma}
\newtheorem{proposition}{Proposition}
\newtheorem{remark}{Remark}
\newtheorem{theorem}{Theorem}
\newtheorem{example}{Example}
\numberwithin{equation}{section}

\makeatletter
\@namedef{subjclassname@2020}{%
  \textup{2020} Mathematics Subject Classification}
\makeatother

\title[Critical metrics of the volume functional]{Critical metrics of the volume functional\\ on complete manifolds}

\author{Caio Coimbra}
\author{Rafael Di\'ogenes}
\author{Ernani Ribeiro Jr}

			\address[C. Coimbra]{Universidade Federal do Cear\'a - UFC, Departamento  de Matem\'atica, Campus do Pici, Av. Humberto Monte, Bloco 914, 60455-760, Fortaleza - CE, Brazil}
	\email{caioadler@alu.ufc.br}

		\address[R. Di\'ogenes]{UNILAB, Instituto de Ci\^encias Exatas e da Natureza, Rua Jos\'e Franco de Oliveira, s/n, 62790-970, Reden\c{c}\~ao - CE, Brazil}\email{rafaeldiogenes@unilab.edu.br}

		\address[E. Ribeiro]{Universidade Federal do Cear\'a - UFC, Departamento  de Matem\'atica, Campus do Pici, Av. Humberto Monte, Bloco 914, 60455-760, Fortaleza - CE, Brazil}
		\email{ernani@mat.ufc.br} 
	
\thanks{C. Coimbra was partially supported by CAPES/Brazil - Finance Code 001}
	
	\thanks{E. Ribeiro Jr was partially supported by CNPq/Brazil Grant 305128/2025-6 and FUNCAP/Brazil Grant ITR-0214-00116.01.00/23.}
	
\thanks{R. Di\'ogenes was partially supported by CNPq/Brazil Grant 305731/2024-6}



\date{\today}

\begin{document}

\begin{abstract}
In this article, we investigate critical metrics of the volume functional on complete manifolds without boundary. We  prove that any cri\-ti\-cal metric of the volume functional on a connected, complete manifold with parallel Ricci tensor is isometric to one of the standard models. Moreover, we show that a Bach-flat cri\-ti\-cal metric of the volume functional on a complete, simply connected manifold with proper potential function is isometric to one of the following: the standard sphere $\mathbb{S}^n$, Euclidean space $\mathbb{R}^n$, hyperbolic space $\mathbb{H}^n$, or a warped product $\mathbb{R} \times_{\varphi} \Sigma_c$, where $\Sigma_c$ is a regular level set of the potential function. In particular, we establish classification results in dimensions three and four under weaker assumptions on the Bach tensor.
\end{abstract}

\maketitle
\section{Introduction}
\label{intro}

A classical approach to studying canonical metrics on a smooth manifold is to analyze the critical points of geometric functionals under appropriate constraints. 
From the seminal works of Einstein and Hilbert, it is well known that the critical points of the total scalar curvature functional, when restricted to the set of Riemannian metrics $g$ on a compact manifold $M^n$ with unit volume, are precisely the Einstein metrics; that is, those satisfying $Ric = \lambda g$, where $\lambda$ is a constant and $Ric$ denotes the Ricci curvature of $(M^n,\,g)$; see \cite[Theorem 4.21]{besse2007einstein}. In the same spirit, it follows from Besse \cite[p. 127–128]{besse2007einstein} that the Euler--Lagrange equation of the total scalar curvature functional, when restricted to metrics of unit volume and constant scalar curvature, takes the form $
\mathcal{L}_g^{*}(f)=-(\Delta f)g+\nabla^2 f-fRic=Ric-\frac{R}{n}g,$ where $\mathcal{L}_g^{*}$ is the formal $L^2$-adjoint of the linearization of the scalar curvature operator $\mathfrak{L}_g,$ $f$ is a smooth function on $M^n$ and $R$ stands for the scalar curvature of $(M^n,\,g).$ Here, $\Delta$ and $\nabla^2$ denote the Laplacian and Hessian, respectively. These foundational results have inspired a rich theory involving other functionals and geometric constraints.

In a related direction, and motivated by a volume comparison theorem due to Fan, Shi, and Tam \cite{fan2007large}, Miao and Tam \cites{miao2011einstein,miao2009volume} initiated the study of critical metrics of the volume functional constrained to the space of metrics with constant scalar curvature on compact manifolds with boundary. Subsequently, Corvino, Eichmair, and Miao \cite{corv2013def} employed this framework to establish a deformation result, which indicates that scalar curvature alone is not sufficient to  establish a volume comparison result. In particular, their deformation result implies that Schoen’s conjecture cannot be directly extended to manifolds with boundary under only Dirichlet boundary conditions (see also \cite{yuan}). Recall that \textit{Schoen’s conjecture}\footnote{It has been confirmed in dimension three by the works of Hamilton \cite{Hamilton} and Perelman \cite{Perelman}, and also for metrics $C^2$-close to the hyperbolic metric, through results by Besson--Courtois--Gallot \cite{Besson1, Besson2}.} \cite{Schoen2} asserts: {\it Let $(M^n,\,\overline{g})$ be a closed hyperbolic manifold and let $g$ be another metric on $M^n$ with scalar curvature $R(g) \geq R(\overline{g}),$ then $Vol(g) \geq Vol(\overline{g})$}. It follows from the work of Corvino, Eichmair and Miao \cite{corv2013def} that, for a given constant $\kappa,$ if a metric $g$ does not admit a nontrivial solution $f$ to the equation
\begin{equation}
\label{eqVS}
\mathcal{L}_g^{*}(f) =-(\Delta f)g + \nabla^2 f-fRic= \kappa g,
\end{equation} then it is possible to simultaneously prescribe a compactly supported perturbation of the scalar curvature within a bounded domain $\Omega$, and a prescribed perturbation of the volume, via a small deformation of the metric supported in $\overline{\Omega}.$ Despite that, scalar curvature and volume comparison results do hold for certain special metrics. For instance, motivated by the work of Brendle, Marques and Neves \cite{BMN} on Min-Oo's conjecture, Miao and Tam \cite{MTCAG} proved a rigidity theorem for the upper hemisphere with respect to nondecreasing scalar curvature and volume. They also showed that an analogous result holds for Euclidean domains. In particular, these spaces satisfy equation~\eqref{eqVS}. More recently, Yuan \cite[Theorem A]{yuan} established a volume comparison result for small geodesic balls under appropriate boundary conditions.

Before proceeding, we fix the following terminology (cf. \cite{corv2013def,miao2011einstein,miao2009volume}).

\begin{definition}
\label{defstatic}
    Let $(M^n,\,g)$ be a complete Riemannian manifold. We say that $g$ is a $V$-static metric if there exists a constant $\kappa$ and a non-constant smooth function $f$ satisfying (\ref{eqVS}). In this case, $f$ is a $V$-static potential. 
\end{definition} 

$V$-static metrics are critical points of the volume functional restricted to the space of metrics with prescribed constant scalar curvature and fixed boundary me\-tric; see \cite[Theorem 2.3]{corv2013def} and \cite[Theorem 2.1]{miao2009volume}. Moreover, as noted by McCormick \cite{McCormick}, they arise in the study of asymptotically hyperbolic manifolds as critical points of the volume-renormalized mass. Observe that when $\kappa = 0$, equation~\eqref{eqVS} reduces to the static vacuum equation, relevant in general relativity (see \cite{ambrozio2017static,corvino2000scalar}). Furthermore, when $f$ is constant, the equation forces the metric to be Einstein. Hence, $V$-static metrics generalize Einstein metrics in a natural way. Interestingly, it follows from \cite[Proposition 2.1]{corv2013def} and \cite[Theorem 7]{miao2009volume} that any connected Riemannian manifold satisfying \eqref{eqVS} must have constant scalar curvature.

There are explicit examples of $V$-static metrics on both compact and noncompact manifolds. These include the Schwarzschild and AdS-Schwarzschild metrics restricted to suitable domains, as well as the standard metrics on geodesic balls in $\mathbb{R}^n$, $\mathbb{H}^n$, and $\mathbb{S}^n$ (cf. \cites{corv2013def,miao2011einstein,miao2009volume}). The classification of $V$-static metrics plays a central role in understanding the interplay between scalar curvature and volume. As noted in \cite{corv2013def,miao2011einstein,miao2009volume}, $V$-static metrics tend to exhibit strong rigidity properties. In recent years, several rigidity results have been established, especially in the compact with boun\-da\-ry case; see, for instance, \cite{barros2015bach,BDR21,baltazar2017critical,BDRIJM,Allan2,corv2013def,DPR,FY,He,KS,miao2011einstein,miao2009volume,SW,yuan}.

In this article, we focus on rigidity phenomena for $V$-static metrics $(M^n,\,g,\,f)$ on complete manifolds without boundary. A key result in this setting is due to Miao and Tam \cite[Theorem 2.2]{miao2011einstein}, who showed that if a $V$-static metric ($\kappa \neq 0$) on a connected, complete manifold without boundary is Einstein, then $(M^n,\,g)$ must be isometric to one of the following:

\begin{itemize}
\item the standard sphere $\mathbb{S}^n,$ 
\item the Euclidean space $\mathbb{R}^n,$ 
\item the hyperbolic space $\mathbb{H}^n,$ or
\item a warped product space $(\mathbb{R}\times \Sigma^{n-1},\,dt^2+ \cosh^2 t g_0)$, where $(\Sigma^{n-1},g_0)$ is complete, Einstein with $Ric_{g_0} = -(n-2)g_0$, and the solution is $f(t,x) = A\,\sinh t + \frac{1}{n-1}$, with $A>0$.
\end{itemize} The result was proved for $\kappa=1,$ but, up to a rescaling of $f,$ it also holds to general $V$-static metrics ($\kappa\neq 0$). We note that the case of connected, compact manifolds with boundary was treated in \cite[Theorem 1.1]{miao2011einstein}.

Since manifolds with parallel Ricci tensor have harmonic curvature (although the converse does not generally hold; see \cite{Derd}), it is natural to ask whether the Einstein condition assumed in \cite[Theorem 2.2]{miao2011einstein} can be relaxed to the weaker assumption of parallel Ricci tensor, that is, $\nabla Ric =0$. This question was answered affirmatively in the case of compact manifolds with boundary by Baltazar--Ribeiro \cite{baltazar2017critical}, who proved that if a critical metric of the volume functional on an $n$-dimensional compact, connected manifold with boundary has parallel Ricci tensor, then the manifold is isometric to a geodesic ball in a simply connected space form $\mathbb{R}^n$, $\mathbb{H}^n$, or $\mathbb{S}^n$. Notably, this result excludes the \textit{Nariai space}—the manifold $I \times \mathbb{S}^n$ with the product metric—as a $V$-static space with $\kappa \neq 0$, even though it is a static vacuum space (i.e., $\kappa = 0$); see \cite{ambrozio2017static,Nariai}.

Our first main result addresses the aforementioned question in the setting of complete manifolds without boundary. More precisely, we prove the following result.

    \begin{theorem}\label{ThmA}
        Let $(M^n,\,g,\,f)$ be an $n$-dimensional connected, complete $V$-static metric with $\kappa \neq 0$ and parallel Ricci tensor. Then $(M^n,\,g)$ is isometric to either 
\begin{enumerate}
\item the standard sphere $\mathbb{S}^n,$ or 
\item the Euclidean space $\mathbb{R}^n,$ or 
\item the hyperbolic space $\mathbb{H}^n$ when $\nabla f(p)=0$ for some $p \in M,$ or  
\item a warped product space $(\mathbb{R}\times\Sigma^{n-1},\,dt^2+\cosh^2 t g_0)$, where $(\Sigma^{n-1},g_0)$ is complete, Einstein with $Ric_{g_0} = -(n-2)g_0$, $f(t,x) =\kappa\left( A\,\sinh t + \frac{1}{n-1}\right)$, and $A>0$ is constant.
            \end{enumerate}
            \end{theorem}

            \begin{remark}
           We emphasize that the condition $\kappa \neq 0$ in Theorem~\ref{ThmA} is essential and cannot be removed. Following the idea outlined in Costa et al. \cite[Example 1]{CDPR}, one sees that the product manifold $\mathbb{S}^{p+1} \times \mathbb{S}^q$, with $q > 1$, also defines a static va\-cuum space with parallel Ricci tensor which fails to be Einstein. In the noncompact setting, a simple calculation shows that the manifold $\mathbb{H}^{p+1}\times\mathbb{H}^q,$ $q>1,$ equipped with the metric $g = g_{\mathbb{H}^{p+1}} + \frac{q-1}{p+1}g_{\mathbb{H}^q}$ and potential function $f= cosh(r)$, where $r(x)$ is the height function on $\mathbb{H}^{p+1},$ also defines a static vacuum space with parallel Ricci tensor. However, this manifold is not Einstein whenever $p+1 \neq q.$
            \end{remark}
            
  \begin{remark}        
The proof of Theorem~\ref{ThmA} is partly inspired by \cite{baltazar2017critical}. Ho\-we\-ver, since $M^n$ has no boundary, it is necessary to derive new differential identities to circumvent the use of integration arguments, yielding a more direct and self-contained proof; see also Corollary \ref{cor2b}.
\end{remark}

For what follows, we recall the definition of the {\it Bach tensor} on a Riemannian manifold $(M^n,\,g).$ Originally introduced in the context of conformal relativity by Bach in the 1920s \cite{bac1921zur}, the Bach tensor plays a key role in conformal geometry. For dimensions $n \geq 4$, it is defined in terms of the Weyl tensor $W_{ijkl}$ by

\begin{equation}\label{bach4}
    B_{ij} = \frac{1}{n-3}\nabla_k\nabla_l W_{ikjl} + \frac{1}{n-2}R_{kl}W_{ikjl}.
\end{equation} In dimension $n=3$, where the Weyl tensor vanishes identically, the Bach tensor is instead defined via the Cotton tensor $C_{ijk}$ as
\begin{equation}\label{bach3}
    B_{ij} = \nabla_kC_{kij}.
\end{equation} A manifold $(M^n,\,g)$ is said to be \textit{Bach-flat} if $B_{ij}=0.$ In the $4$-dimensional compact case, Bach-flat metrics arise as the critical points of the conformally invariant functional
\begin{equation*}
    \mathcal{W}(g) = \int_M|W_g|^2 dV_g,
\end{equation*} where $W_g$ denotes the Weyl tensor associated to the metric $g.$ We recall that both locally conformally flat metrics and Einstein metrics are Bach-flat. Moreover, in dimension $n=4,$ it is well known that metrics which are half-conformally flat or locally conformal to Einstein metrics also satisfy the Bach-flat condition.

Inspired by the works of Cao and Chen \cite{caolocally2012,cao2013bach}, Barros, Di\'ogenes and Ribeiro \cite{barros2015bach} proved that a
Bach-flat critical metric of the volume functional on a simply connected $4$-dimensional manifold with boundary isometric to a standard sphere must be isometric to a geodesic ball in a simply connected space form $\Bbb{R}^4,$ $\Bbb{H}^4,$ or $\Bbb{S}^4;$ see also \cite{BBK}. Roughly speaking, they replaced the assumption of locally conformally flat in the Miao--Tam result (cf. \cite[Theorem 1.2]{miao2011einstein}) by the weaker condition of Bach-flat (see \cite{barros2015bach}). Moreover, they established a similar rigidity result in dimension $n=3$ under a weaker assumption.

A natural question that arises from these comments is what happens when the manifold is complete and has no boundary. In this article, we also address that question. More precisely, we esta\-blish the following result.

\begin{theorem}\label{ThmB}
    Let $(M^n,\,g,\,f),$ $n\geq 4,$ be an $n$-dimensional complete, simply connected Bach-flat $V$-static metric with $\kappa\neq 0$ and $f$ be a proper function. Then $(M^n,\,g)$ is isometric to either 
    
\begin{enumerate}
\item the standard sphere $\mathbb{S}^n,$ or 
\item the Euclidean space $\mathbb{R}^n,$ or 
\item the hyperbolic space $\mathbb{H}^n$ or
\item $\mathbb{R}\times_{\varphi}\Sigma_c,$
where $\Sigma_c$ is a regular level set of the potential function, which is an Einstein mani\-fold with $Ric_{\Sigma_c}=\lambda g_{\Sigma_c},$ and the warping function $\varphi$ is a solution of the ODE
\begin{equation}
\label{ODEgeral}
     \varphi\Big(2\varphi'' +\frac{R}{n-1}\varphi\Big) + (n-2)(\varphi')^2 = \lambda.
\end{equation}
\end{enumerate} 
\end{theorem}

\vspace{0.20cm}

\begin{remark}
We emphasize that Theorem~\ref{ThmB} remains valid if the Bach-flat condition is replaced by the weaker assumption that the Bach tensor is radially flat, that is, $B(\nabla f,\nabla f)=0;$ see Lemma \ref{lem5} and Theorem \ref{ThmD} in Section \ref{proofBandC}.
\end{remark}

In order to prove Theorem~\ref{ThmB}, which is a special case of Theorem \ref{ThmD}, we adapt some techniques outlined in~\cite{barros2015bach,catino,cao2013bach, CaoYu, catino2016geometry,lopez2014note}. Our analysis begins by establishing a connection between the conditions $T \equiv 0$ and  Bach-flat, where $T$ denotes an auxiliary $3$-tensor (Lemma~\ref{lem5}). This step required revisiting and adapting several arguments originally formulated in~\cite{barros2015bach}, where the presence of a boundary was essential. In our context, the assumption that the potential function $f$ is proper plays a crucial role in circumventing the absence of boundary conditions. This requirement is also natural because, in our setting, there is no “integrability condition’’ analogous to Hamilton’s identity for Ricci solitons (see Eq. (4.13) in \cite{Book}), which makes it substantially more difficult to control the asymptotic behavior of the potential function. Next, we construct a local warped pro\-duct structure around the regular points of $f$ (Proposition~\ref{Prop2}). To extend this local structure to a global one, we adapt techniques from a recent work of Cao–Yu~\cite{CaoYu}, along with methods introduced by Cao--Chen \cite{cao2013bach} and Fern\'andez-L\'opez and Garc\'ia-R\'io~\cite{lopez2014note} in the study of Ricci solitons, which are well suited to the complete without boundary case. 

\vspace{0.20cm}

We now turn our attention to the lower-dimensional cases. In  particular, when $n=4,$ equation~\eqref{ODEgeral} becomes more tractable. As a consequence of Theorem~\ref{ThmB}, we obtain the following corollary.

\begin{corollary}\label{Cor1}
    Let $(M^4,\,g,\,f)$ be a complete, simply connected four-dimensional Bach-flat $V$-static metric with $\kappa\neq 0$ and $f$ be a proper function. Then $(M^4,\,g)$ is isometric to either 
    
    \begin{enumerate}
\item the standard sphere $\mathbb{S}^4,$ or 
\item the Euclidean space $\mathbb{R}^4,$ or 
\item the hyperbolic space $\mathbb{H}^4,$ or  
\item the warped product $\Bbb{R}\times_{\varphi} \Bbb{S}^3,$ where $\varphi$ is a solution of the ODE $$\varphi\left(\varphi'' + \frac{R}{6}\varphi\right) + (\varphi')^2 =1.$$
\end{enumerate}
\end{corollary}

It is natural to ask whether an analogous result holds in the three-dimensional case. In this situation, however, arguments from Cao et al. \cite{catino}, combined with key steps from the second part of the proof of Theorem~\ref{ThmB}, allows us to establish our next result under the weaker assumption
 $div\,B(\nabla f)=0.$ 

\begin{theorem}
\label{ThmC}
     Let $(M^3,\,g,\,f)$ be a complete, simply connected $V$-static metric with $\kappa\neq 0$ and $f$ be a proper function. If $div\,B(\nabla f)=0,$ then $(M^3,\,g)$ is isometric to either 
    
    \begin{enumerate}
\item the standard sphere $\mathbb{S}^3,$ or 
\item the Euclidean space $\mathbb{R}^3,$ or 
\item the hyperbolic space $\mathbb{H}^3,$ or  
\item the warped product $\mathbb{R}\times_{\varphi}\Bbb{S}^2,$
where $\varphi$ is a solution of the ODE
\begin{equation}
\label{aaaaa}
     \varphi\Big(2\varphi''+\frac{R}{2}\varphi\Big) + (\varphi')^2 = 1.
\end{equation}
\end{enumerate} 
\end{theorem} 

\vspace{0.30cm}

The remainder of this paper is organized as follows. In Section~\ref{prelim}, we review the essential background material and key results on $V$-static metrics that will be used throughout the proofs of our main results. Section~\ref{proofthmA} is dedicated to the proof of Theorem~\ref{ThmA}. In Section~\ref{proofBandC}, we present the proofs of Theorem~\ref{ThmB}, Corollary~\ref{Cor1} and Theorem~\ref{ThmC}.

\section{Background}
\label{prelim}
In this section, we recall some basic tensors from Riemannian geometry and highlight useful lemmas and facts on 
$V$-static spaces, which will play a fundamental role in the proofs of the main theorems.

We begin recalling that a $V$-static metric $(M^n,g,f)$ satisfies
\begin{equation}\label{eqdefsec2}
    -(\Delta f)g + \nabla^2f - fRic = \kappa g.
\end{equation} In particular, we may write (\ref{eqdefsec2}) in the tensorial notation as
\begin{equation}
\label{eqdefsec2aaa}
   -(\Delta f)g_{ij} + \nabla_{i}\nabla_{j} f - f R_{ij} = \kappa g_{ij}.
\end{equation} Taking the trace in (\ref{eqdefsec2aaa}), one sees that 
\begin{equation}\label{eqlap}
    \Delta f = -\frac{fR + \kappa n}{n-1}.
\end{equation} Moreover, it is easy to check that 
\begin{equation}\label{equivRicHess}
    f\mathring{Ric} = \mathring{\nabla^2 f},
\end{equation} where $\mathring{Z}=Z-\frac{{\rm tr}\,Z}{n}g$ stands for the traceless of tensor $Z.$

\begin{remark}
It should be mentioned that, choosing appropriate coordinates, $f$ and $g$ are analytic; see Proposition 2.1 in \cite{corv2013def}. Consequently, the set of regular points of $f$ is dense in $M^n.$
\end{remark}

We now recall some examples of $V$-static metrics obtained by Miao--Tam \cite{miao2011einstein} (see also \cite{corv2013def}) for complete manifolds without boundary. Let us start with the Euclidean space $\mathbb{R}^n$ with its canonical metric.

\begin{example}[\cite{miao2011einstein}]
    Consider $(\mathbb{R}^n,\,\delta)$, where $\delta$ is its canonical metric. Define the function
    \begin{equation}
        f(x) = \frac{1}{(n-1)}\left(A - \frac{\kappa}{2}|x|^2\right),
    \end{equation} where $A$ is a constant. It is straightforward to verify that the triple $(\mathbb{R}^n,\,\delta,\,f)$ satisfies Definition \ref{defstatic} with $\kappa\neq 0.$
\end{example}

We now present a similar example as before on the standard sphere $\mathbb{S}^n.$

\begin{example}[\cite{miao2011einstein}]
\label{exS}
    Consider $(\mathbb{S}^n,\,g)$ the sphere equipped with its canonical metric $g.$ Define the function $f$ along a geodesic  $\gamma(s)$ emanating from a point $p\in \mathbb{S}^n$ by
    \begin{equation}
        f(\gamma(s)) = \frac{1}{n-1}\Big(A \cos\,r-\kappa\Big),
    \end{equation} where $r$ denotes the geodesic distance from the point $p,$ such that $\nabla f(p)=0$, and $A$ is a constant. Thus, $(\mathbb{S}^n,\,g,\,f)$ is a $V$-static  metric.
\end{example}

Reasoning as in the spherical case, we have the following example on the hyperbolic space $\mathbb{H}^n.$

\begin{example}[\cite{miao2011einstein}]
    Consider $(\mathbb{H}^n,\,g_{_{\mathbb{H}^n}})$ the hyperbolic space with its canonical metric and define the function
    \begin{equation}
      f(\gamma(s))=  \frac{1}{n-1}\Big(\kappa -A\cosh\,r\Big),
    \end{equation}
    where $\gamma(s)$ is a geodesic emanating from a fixed point $p\in \mathbb{H}^n$ such that $\nabla f(p)=0$, and $A$ is a constant. Hence, $(\mathbb{H}^n,\,g_{_{\mathbb{H}^n}},\,f)$ is a $V$-static metric.
\end{example}

The next example is constructed using a warped product metric.

\begin{example}[\cite{miao2011einstein}]
\label{example4}
Consider $\mathbb{R}\times\Sigma^{n-1}$ equipped with the warped product metric $g=dt^2 + \cosh^2 t g_0,$ where $(\Sigma^{n-1},g_0)$ is complete, Einstein with $Ric_{g_0} = -(n-2)g_0$, and define the potential function $f(t,x) =\frac{\kappa}{n-1}\left( A\,\sinh t + 1\right)$, and $A>0$ is constant. Thus, $(\mathbb{R}\times\Sigma^{n-1},\,g,\, f)$ is a $V$-static metric.
\end{example}

In the sequel, we establish a few properties of compact $V$-static manifolds without boundary. As a first step, we shall show that such manifolds must necessarily have positive scalar curvature (cf. \cite[Theorem 7]{miao2009volume}).

\begin{proposition}
    Let $(M^n,\,g,\,f)$ be a compact $V$-static manifold without boundary. Then $(M^n,\,g)$ has positive scalar curvature.
\end{proposition}

\begin{proof}
Since $(M^n,\,g)$ has constant scalar curvature, we first assume that $R=0.$ In this case, it follows from (\ref{eqlap}) that $(n-1)\Delta f = -\kappa n$ and hence, by the Maximum Principle, $f$ is constant. Now, suppose that $R<0.$ Thus, we can choose points $p,q \in M$ such that $f(q) = \min f$ and $f(p) = \max f.$ Then, one sees that
    \begin{equation*}
        -Rf(q) - \kappa n\le -Rf - \kappa n\le -Rf(p)- \kappa n,
    \end{equation*} and by (\ref{eqlap}), we have
    \begin{equation*}
        0\leq \Delta f(q) \le \Delta f \le \Delta f(p)\leq 0.
    \end{equation*} Thus, $\Delta f =0$ and therefore, $f$ must be constant. This finishes the proof. 
\end{proof}

Next, we shall show that if the potential function of a compact $V$-static manifold without boundary does change sign, then it is necessarily Example \ref{exS}.

\begin{proposition}
    Let $(M^n,\,g,\,f)$ be a compact $V$-static manifold without boundary. Suppose that $f$ does not change sign. Then $(M^n,\,g)$ is isometric to the standard sphere $\mathbb{S}^{n}$.
\end{proposition}
\begin{proof}
    Notice that
    \begin{equation*}
        div(\mathring{Ric}(\nabla f)) = (div\mathring{Ric})(\nabla f) + \langle \mathring{Ric},\mathring{\nabla^2 f}\rangle.
    \end{equation*} By the twice contracted second Bianchi identity $(2 div\,Ric=\nabla R),$ one sees that $div\mathring{Ric}=0.$ Hence, it follows from (\ref{equivRicHess}) that
    \begin{equation}
    \label{ekl1110}
        div(\mathring{Ric}(\nabla f)) = f|\mathring{Ric}|^2.
    \end{equation}
   Upon integrating (\ref{ekl1110}) over $M^n$, we apply the Stokes' theorem combined with the fact that $M^n$ has no boundary in order to infer
    \begin{equation*}
        \int_M f|\mathring{Ric}|^2 dV_g =0.
    \end{equation*} Since $f$ does not change sign, we deduce that $\mathring{Ric}\equiv0,$ and hence the manifold is Einstein. Therefore, by applying \cite[Theorem 2.2]{miao2011einstein}, we conclude that $(M^n,\,g)$ is isometric to the standard sphere $\mathbb{S}^n,$ as stated. 
    \end{proof}

    \subsection{Some key tensors}
We recall some classical tensors that play a fundamental role in the study of Riemannian manifolds. The first is the Weyl tensor $W,$ which is the traceless component of the Riemann curvature tensor $R_{ijkl}$ of a Riemannian manifold $(M^n,\,g).$ It is defined by the following decomposition:
\begin{eqnarray}\label{eqW}
     R_{ijkl} &=& W_{ijkl} +\frac{1}{n-2}\Big(R_{ik}g_{jl}+ R_{jl}g_{ik} -R_{il}g_{jk}-R_{jk}g_{il}\Big) \\\nonumber
    &&-\frac{R}{(n-1)(n-2)}\Big(g_{jl}g_{ik}-g_{il}g_{jk}\Big).
\end{eqnarray}
Next, we consider the Cotton tensor $C_{ijk},$ which is defined as
\begin{equation}\label{eqC}
    C_{ijk} = \nabla_iR_{jk}-\nabla_jR_{ik} - \frac{1}{2(n-1)}\Big(\nabla_iRg_{jk}- \nabla_jRg_{ik}\Big).
\end{equation} For dimensions $n\geq 4,$ the Cotton tensor is related to the Weyl tensor by
\begin{equation}\label{eqC2}
    C_{ijk} = -\frac{n-2}{n-3}\nabla_lW_{ijkl}.
\end{equation} Another important tensor is the Schouten tensor, defined by
\begin{equation}\label{eqA}
    A_{ij} = R_{ij}-\frac{R}{2(n-1)}g_{ij}.
\end{equation} In terms of $A,$ we can express the Riemann curvature tensor as   
\begin{equation}
    R_{ijkl} = \frac{1}{n-2}\left(A\odot g\right)_{ijkl} + W_{ijkl}.
\end{equation}

Now, in the context of a $V$-static metric $(M^n,\,g,\,f),$ it is useful to recall an auxiliary $3$-tensor $T$ introduced by Barros--Di\'ogenes--Ribeiro \cite{barros2015bach} and inspired by the tensor $D$ defined by Cao--Chen \cite{cao2013bach} in the study of Ricci solitons. More precisely, given an $n$-dimensional $V$-static metric $(M^n,\,g,\,f),$ the tensor $T_{ijk}$ is given by

\begin{eqnarray}\label{eqT}
    T_{ijk} &=& \frac{n-1}{n-2}\Big(R_{ik}\nabla_j f - R_{jk}\nabla_i f\Big) - \frac{R}{n-2}\Big(g_{ik}\nabla_jf - g_{jk}\nabla_if\Big)\nonumber\\
    &&+\frac{1}{n-2}\Big(g_{ik}R_{js}\nabla_s f - g_{jk}R_{is}\nabla_s f\Big).
\end{eqnarray} Notice that $T_{ijk}$ is skew-symmetric in the first two indices and trace-free in any two indices.

It turns out that the tensor $T$ can be expressed in terms of the Weyl and Cotton tensors as follows (cf. \cite[Lemma 2]{barros2015bach}).

\begin{lemma}[\cite{barros2015bach}]
\label{lem2}
    Let $(M^n,\,g,\,f)$ be an $n$-dimensional $V$-static metric. Then we have:
    \begin{equation}
        fC_{ijk} = T_{ijk} + W_{ijkl}\nabla_l f.
    \end{equation}
\end{lemma}

The tensor $T$ is closely related to the geometry of the level surfaces $\Sigma_c$ of the potential function $f,$ namely, $\Sigma_c = \{x\in M :\,f(x)=c\}.$ Assuming that $T$ vanishes identically, we obtain the following interesting properties (cf. \cite[Proposition 1]{barros2015bach}).

\begin{proposition}[\cite{barros2015bach}]
\label{Prop1}
    Let $(M^n,\,g,\,f)$ be an $n$-dimensional $V$-static metric with $T\equiv 0$. Let $c$ be a
regular value of $f$ and $\Sigma_c = \{x\in M :\,f(x)=c\}$ be a level set of $f.$ Consider $e_1 := \frac{\nabla f}{|\nabla f|}$ and choose an orthonormal frame $\{e_2,...,e_n\}$ tangent to $\Sigma_c$. Then the following assertions hold:

    \begin{enumerate}
        \item The second fundamental form $h_{\alpha\beta}$ of $\Sigma_c$ is $h_{\alpha\beta} = \frac{H}{n-1}g_{\alpha\beta}$;
        \item $|\nabla f|$ is constant on $\Sigma_c$;
        \item $R_{1\alpha} = 0$ for any $\alpha \geq 2$ and $e_1$ is an eigenvector of $Ric$;
        \item The mean curvature of $\Sigma_c$ is constant;
        \item On $\Sigma_c$, the Ricci tensor either has a unique eigenvalue or two distinct eigenvalues with multiplicity $1$ and $n-1,$ respectively. Moreover, the eigenvalue with multiplicity $1$ is in the direction of $\nabla f$;
        \item $R_{1\alpha\beta\gamma}=0$, $\alpha,\beta,\gamma \in \{2,...,n\}.$
    \end{enumerate}
\end{proposition}

\section{$V$-static metrics with parallel Ricci tensor}
\label{proofthmA}

This section is devoted to the proof of Theorem~\ref{ThmA}. To begin with, we need to recall the following lemma, first established in~\cite[Lemma 1]{barros2015bach}.

\begin{lemma}[\cite{barros2015bach}]
\label{lem1}
    Let $(M^n,\,g,\,f)$ be an $n$-dimensional $V$-static metric. Then we have:
    \begin{equation*}
        f(\nabla_i R_{jk} - \nabla_j R_{ik}) = R_{ijkl}\nabla_l f + \frac{R}{n-1}(\nabla_i f g_{jk}-\nabla_j f g_{ik})-(\nabla_i f R_{jk}-\nabla_j f R_{ik}).
           \end{equation*}
\end{lemma}

\vspace{0.30cm}
We now present the proof of Theorem \ref{ThmA}.

\begin{proof}[{\bf Proof of Theorem \ref{ThmA}}]
    Since $M^n$ has parallel Ricci tensor, we have $\nabla_iR_{jk} = 0$ for all indices $1 \leq i,j,k \leq n.$ Hence, by using the Ricci identity, one obtains that
\begin{align*}
    0 &= f(\nabla_l\nabla_kR_{lj})R_{kj} = f(\nabla_k\nabla_lR_{lj} + R_{lkli}R_{ij} + R_{lkji}R_{li})R_{kj}\\
    &=fR_{ki}R_{ij}R_{kj} + fR_{lkji}R_{li}R_{kj}.
\end{align*} Rearranging the indices, we get
\begin{equation}\label{eq1}
    fR_{ij}R_{ik}R_{kj} = fR_{ijkl}R_{jl}R_{ik}.
\end{equation} 

On the other hand, by using (\ref{eqdefsec2aaa}), one sees that

\begin{align*}
    fR_{ijkl}R_{jl}R_{ik} =& R_{ijkl}R_{jl}\left(-(\Delta f + \kappa)g_{ik} + \nabla_i\nabla_k f\right)\\
    =&-(\Delta f + \kappa)|Ric|^2 + R_{ijkl}R_{jl}\nabla_{i}\nabla_{k}f\\ =&-(\Delta f + \kappa)|Ric|^2 + \nabla_i(R_{ijkl}R_{jl}\nabla_k f)-\nabla_iR_{ijkl}R_{jl}\nabla_k f,
\end{align*} where we used that $\nabla_iR_{jl}=0.$ It follows from the Bianchi identity that

\begin{equation}
\label{secBia}
(div\,Rm)_{jkl}=\nabla_{k}R_{jl} - \nabla_{l}R_{jk}.
\end{equation} Hence, our assumption implies $div\,Rm=0.$ Consequently,

\begin{equation}\label{eq211}
    fR_{ijkl}R_{jl}R_{ik}= -(\Delta f + \kappa)|Ric|^2 + \nabla_i(R_{ijkl}R_{jl}\nabla_k f).
\end{equation}

Notice that 
\begin{eqnarray*}
\nabla_{i}\left(\nabla_{j}f R_{ik}R_{jk}\right)=\nabla_{i}\nabla_{j}f R_{ik} R_{jk}, 
\end{eqnarray*} and using (\ref{eqdefsec2aaa}), we have 

\begin{equation*}
\nabla_{i}\left(\nabla_{j}f R_{ik}R_{jk}\right)=(\Delta f + \kappa)|Ric|^2 +f R_{ij}R_{ik}R_{jk}.
\end{equation*} This jointly with (\ref{eq211}) gives

\begin{equation}\label{eq4}
    \nabla_i(R_{ijkl}R_{jl}\nabla_k f) = \nabla_i(\nabla_j fR_{ik}R_{jk}).
\end{equation} 

From Lemma \ref{lem1} and our assumption, one sees that

\begin{align*}
    R_{ijkl}\nabla_k fR_{jl} =& \bigg(\frac{R}{n-1}\big(\nabla_i fg_{jl}-\nabla_jf g_{il}\big) - \big(\nabla_i fR_{jl}-\nabla_j fR_{il}\big)\bigg)R_{jl}\\
    =& \frac{R^2}{n-1}\nabla_i f - \frac{R}{n-1}\nabla_j fR_{ij}- \nabla_i f|Ric|^2 + \nabla_j fR_{il}R_{jl},
\end{align*} so that

\begin{eqnarray}\label{eq5}
    \nabla_i(R_{ijkl}\nabla_k fR_{jl}) &=& \frac{R^2}{n-1}\Delta f - \frac{R}{n-1}\nabla_i\nabla_j fR_{ij} - \Delta f|Ric|^2\nonumber\\&& + \nabla_i(\nabla_j fR_{ik}R_{jk}),
\end{eqnarray} where we used again the Ricci parallel condition. By combining (\ref{eq5}) and (\ref{eq4}), one obtains that 

\begin{equation}
\label{eq991}
0=\frac{R^2}{n-1}\Delta f - \frac{R}{n-1}\nabla_i\nabla_j fR_{ij} - \Delta f|Ric|^2.
\end{equation} Now, it follows from (\ref{eqdefsec2aaa}) that
\begin{equation*}
    \nabla_i\nabla_j fR_{ij} = f|Ric|^2 + (\Delta f+\kappa)R, 
\end{equation*} which combined with (\ref{eq991}) yields

\begin{eqnarray}
\label{plkl10}
    0 &= &\frac{R^2}{n-1}\Delta f -\frac{fR}{n-1}|Ric|^2 - \frac{R^2(\Delta f+\kappa)}{n-1} - \Delta f|Ric|^2\nonumber\\
    &=&\frac{\kappa n}{n-1}\bigg(|Ric|^2 - \frac{R^2}{n}\bigg),
\end{eqnarray}
where in the second equality we used (\ref{eqlap}). Since $\kappa\neq 0,$ one obtains that $$\left|Ric-\frac{R}{n}g\right|^2 =|Ric|^2 - \frac{R^2}{n}=0$$ and therefore, $(M^n,\,g)$ is an Einstein manifold. Thus, it suffices to apply \cite[Theorem 2.2]{miao2011einstein} to conclude that $(M^n,\,g)$ is isometric to either the standard sphere $\mathbb{S}^n,$ or the Euclidean space $\mathbb{R}^n,$ or the hyperbolic space $\mathbb{H}^n,$ or a warped product space $(\mathbb{R}\times\Sigma^{n-1},\,dt^2+\cosh^2 t g_0)$, where $(\Sigma^{n-1},g_0)$ is complete, Einstein with $Ric = -(n-2)g_0$, $f(t,x) =\kappa\left( A\,\sinh t + \frac{1}{n-1}\right)$, and $A>0$ is constant. So, the proof is completed. 
\end{proof}

Another advantage of this approach is that the proof of Theorem~\ref{ThmA} yields a new and more direct proof of \cite[Corollary~1]{baltazar2017critical}. More precisely, we get the following corollary.

\begin{corollary}
\label{cor2b}
Let $(M^n,\, g,\, f)$ be an $n$-dimensional connected, compact $V$-static metric with $\kappa\neq 0,$ parallel Ricci tensor and smooth boundary $\partial M.$ Then $(M^n,\, g)$ is
isometric to a geodesic ball in a simply connected space form $\mathbb{R}^n$, $\mathbb{H}^n$, or $\mathbb{S}^n$. 
\end{corollary}

\begin{proof}
By following the same steps as in the proof of Theorem~\ref{ThmA} up to equation~(\ref{plkl10}), we conclude that $(M^n,\,g)$ is an Einstein manifold. It then suffices to apply Theorem 1.1 in \cite{miao2011einstein} to deduce that $(M^n,\,g)$ is isometric to a geodesic ball in a simply connected space form $\mathbb{R}^n$, $\mathbb{H}^n$, or $\mathbb{S}^n$.  
\end{proof}

\section{Bach-flat $V$-static metrics}
\label{proofBandC}

In this section, we discuss results concerning Bach-flat $V$-static metrics. Speci\-fi\-cally, we present the proofs of Theorem~\ref{ThmB}, Corollary~\ref{Cor1}, and Theorem~\ref{ThmC}.

We begin by establishing the following key lemma, which is analogous to \cite[Lemma 5]{barros2015bach} and \cite[Lemma 4.1]{cao2013bach}. However, unlike the case of Ricci solitons, in our setting there is no ``integrability condition" analogous to Hamilton’s equation for Ricci solitons (see Eq. (4.13) in \cite{Book}), which poses an obstacle to determining the asymptotic behavior of the potential function $f.$ Therefore, we shall assume that $f$ is proper.

\begin{lemma}\label{lem5}
    Let $(M^n,\,g,\,f),$ $n\geq 4,$ be an $n$-dimensional $V$-static metric and $f$ be a proper function. Then we have:
    \begin{equation}
        \int_{M}f^2B(\nabla f,\nabla f)\, dV_g = - \frac{1}{2(n-1)}\int_{M}f^2|T|^2\, dV_g,
    \end{equation}
    where $B$ is the Bach tensor.
\end{lemma}

\begin{proof} First of all, we combine (\ref{bach4}) and (\ref{eqC2}) to infer
    \begin{equation}
        B_{ij} = \frac{1}{n-2}\left(\nabla_k C_{kij} + R_{kl}W_{ikjl}\right).
    \end{equation} Consequently,
    \begin{align*}
        f^2B_{ij} &= \frac{1}{n-2}\Big(f^2\nabla_kC_{kij} + f^2R_{kl}W_{ikjl}\Big)\\
        &=\frac{1}{n-2}\Big(\nabla_k(f^2C_{kij}) - 2fC_{kij}\nabla_k
f + f^2R_{kl}W_{ikjl}\Big).
\end{align*}

Now, using Lemma \ref{lem2} and (\ref{eqC2}) one obtains that
\begin{align*}
    f^2B_{ij} &= \frac{1}{n-2}\Big(\nabla_k[f(W_{kijl}\nabla_l f +T_{kij})] - 2fC_{kij}\nabla_kf + f^2R_{kl}W_{ikjl}\Big)\\
    &=\frac{1}{n-2}\Big(\nabla_k(fT_{kij})+f(\nabla_kW_{kijl})\nabla_lf + fW_{kijl}\nabla_k\nabla_lf\\
    &\quad+W_{kijl}\nabla_kf\nabla_lf - 2fC_{kij}\nabla_kf + f^2R_{kl}W_{ikjl}\Big)\\
    &=\frac{1}{n-2}\Big(\nabla_k(fT_{kij}) + \frac{n-3}{n-2}fC_{jki}\nabla_kf+f(\nabla_k\nabla_lf-fR_{kl})W_{kijl}\\
    &\quad+W_{kijl}\nabla_kf\nabla_lf + 2fC_{ikj}\nabla_kf\Big).
\end{align*} Using (\ref{eqdefsec2aaa}) and the fact that the Weyl tensor is trace-free in any pair of indices, we obtain

\begin{equation*}
    f^2B_{ij} = \frac{1}{n-2}\Big(\nabla_k(fT_{kij}) + \frac{n-3}{n-2}fC_{jki}\nabla_kf+2fC_{ikj}\nabla_kf + W_{kijl}\nabla_kf\nabla_lf\Big).
\end{equation*} Hence, 

\begin{equation}
\label{plk7859}
    f^2 B(\nabla f,\nabla f) = \frac{1}{n-2}\nabla_k(fT_{kij})\nabla_if\nabla_jf.
\end{equation}

Proceeding, we analyze the right hand side of (\ref{plk7859}). Observe that

\begin{align*}
    \nabla_k(fT_{kij})\nabla_if\nabla_jf &= \nabla_k(fT_{kij}\nabla_if\nabla_jf) - fT_{kij}\nabla_k\nabla_if\nabla_jf - fT_{kij}\nabla_if\nabla_k\nabla_jf\\
    &=\nabla_k(fT_{kij}\nabla_if\nabla_jf)-fT_{kij}(fR_{ki}+ (\kappa+\Delta f)g_{ik})\nabla_jf\\
    &\quad-fT_{kij}\nabla_if(fR_{kj} + (\kappa + \Delta f)g_{kj})\\
    &=\nabla_k(fT_{kij}\nabla_if\nabla_jf) - f^2T_{kij}R_{ki}\nabla_jf
 - f^2T_{kij}R_{kj}\nabla_if,
 \end{align*} where we have used (\ref{eqdefsec2aaa}) and the fact that $T$ is trace-free. Rearranging the indices and using the antisymmetry of $T$, we get
 \begin{equation*}
    (n-2)f^2B(\nabla f,\nabla f) = \nabla_k(fT_{kij}\nabla_if\nabla_jf) - \frac{1}{2}f^2T_{kij}(R_{kj}\nabla_if-R_{ik}\nabla_jf).
\end{equation*} So, it follows from (\ref{eqT}) that
\begin{equation}
\label{43}
    (n-2)f^2B(\nabla f,\nabla f) = \nabla_k(fT_{kij}\nabla_if\nabla_jf) -\frac{n-2}{2(n-1)}f^2|T|^2.
\end{equation}

Notice that if $M^n$ is compact, the result follows immediately from the divergence theorem. Otherwise, if $M^n$ is not compact, we consider the set $$\Omega_r = \{p\in M:\,\, |f(p)| \leq r\}.$$ Taking into account that $f$ is a proper function, we apply the divergence theorem for $\Omega_r$ in order to infer

\begin{equation*}
    \int_{\Omega_r}\nabla_k(fT_{kij}\nabla_if\nabla_jf) \,d\Omega = \int_{\partial\Omega_r}fT_{kij}\nabla_if\nabla_jfN_k\,dS,
\end{equation*}
where $N$ is the unitary normal vector field over $\partial\Omega_r$, i.e., $N = \frac{\nabla f}{|\nabla f|}$ or $N = -\frac{\nabla f}{|\nabla f|}$. But, using again (\ref{eqT}), one concludes that

\begin{equation*}
    \int_{\partial\Omega_r}fT_{kij}\nabla_if\nabla_jfN_k \,dS= \pm\int_{\partial\Omega_r}fT_{kij}\nabla_if\nabla_jf\frac{\nabla_kf}{|\nabla f|}\,dS = 0. 
\end{equation*} Therefore, on integrating (\ref{43}) over $\Omega_{r}$ and using these facts, we obtain

\begin{equation}
 \int_{\Omega_r}f^2B(\nabla f,\nabla f)\,d\Omega = - \frac{1}{2(n-1)}\int_{\Omega_r}f^2|T|^2\, d\Omega.
\end{equation}

Finally, since $f$ is a proper function, $\Omega_r$ is an exhaustion of $M^n$ and hence, letting $r\to \infty,$ we get
\begin{equation}
    \int_{M}f^2B(\nabla f,\nabla f) dV_{g}= - \frac{1}{2(n-1)}\int_{M}f^2|T|^2\,dV_{g},
\end{equation} as asserted.
\end{proof}

Before proceeding, we need to establish an additional structural property of $V$-static metrics. The following proposition, inspired by \cite[Lemma 4]{barros2015bach} and \cite[Theorem 1.2]{catino2016geometry}, highlights a key geometric implication of the condition $T \equiv 0$ in the study of $V$-static metrics.

 \begin{proposition}\label{Prop2}
    Let $(M^n,\,g,\,f),$ $n \ge3,$ be a complete, noncompact $V$-static metric with $\kappa\neq0.$ Suppose that the tensor $T$ vanishes identically. Then, in a neighborhood of every regular level set of $f$, the manifold $(M^n,\,g)$ is locally a warped product with $(n-1)$-dimensional Einstein fibers.
\end{proposition}

\begin{proof}
Since $\Sigma_c = \{p\in M; f(p) = c\}$ is the regular level set of the potential function $f$ and $T=0$ on $M,$ assertion (2) of Proposition \ref{Prop1} ensures that $|\nabla f|$ is constant along $\Sigma_c$. Consequently, in a neighborhood $U$ of $\Sigma_c$ containing no critical points of $f,$ the function $f$ depends only on the signed distance $r$ to the hypersurface $\Sigma_c.$ Therefore, after a suitable change of variables, the metric  $g$ can be locally expressed in the form

    \begin{equation*}
        g = dr^2 + g_{\alpha\beta}(r,\theta)d\theta^\alpha \otimes d\theta^\beta,
    \end{equation*}
    where $\{\theta^2,\dots,\theta^n\}$ is a local coordinate system on $\Sigma_c,$ and $r$ is defined on a maximal interval $r \in (r_{min},r_{max})$ with $r_{min} \in [-\infty,0)$ and $r_{max}\in(0,\infty]$.

     Recall that Proposition \ref{Prop1} asserts that $\Sigma_c$ is a totally umbilical hypersurface with constant mean curvature. Hence, we have
    \begin{equation*}
        \partial_rg_{\alpha\beta} = -2h_{\alpha\beta} = \phi(r)g_{\alpha\beta},
    \end{equation*} with $\phi(r) = -\frac{2}{n-1}H(r)$ and $\partial_{r}=\frac{\nabla f}{|\nabla f|};$ see also Eq. (2.24) in \cite[]{barros2015bach}. From this, it follows that 
    \begin{equation*}
        g_{\alpha\beta}(r,\theta) = e^{\Phi(r)}g_{\alpha\beta}(0,\theta),\,\,\,\hbox{where}\,\,\, \Phi(r) = \int_{0}^r\phi(r)dr.
    \end{equation*} Therefore, in $U$, the metric $g$ takes the form of a warped product
    \begin{equation*}
       g = dr^2 + \varphi(r)^2g_{\Sigma_c},
    \end{equation*}
    where $\varphi$ is a smooth function on $U$, $r \in (r_{min},r_{max})$ and $g_{\Sigma_c}$ is the metric of the regular level set $\Sigma_c.$ To conclude, since $(U,g,\bar{f)}$, where $\bar{f} = f|_{U}$, is a $V$-static warped metric, it suffices to apply \cite[Proposition 3.1 (i)]{miao2011einstein} to deduce that the fiber is Einstein $Ric_{\Sigma_C}=(n-2)k_{0}\,g_{\Sigma_C},$ where $k_0$ is constant; see also the proof of Lemma 4 in \cite{barros2015bach}. This completes the proof of the proposition.
   
\end{proof}

We are now ready to prove {\bf Theorem~\ref{ThmB}}, which follows as a special case of the following more general result.

\begin{theorem}\label{ThmD}
    Let $(M^n,\,g,\,f),$ $n\geq 4,$ be a complete, simply connected $V$-static metric with $\kappa\neq 0$ and $f$ be a proper function. If $B(\nabla f,\nabla f)=0$, then $(M^n,\,g)$ is isometric to either 
    
\begin{enumerate}
\item the standard sphere $\mathbb{S}^n,$ or 
\item the Euclidean space $\mathbb{R}^n,$ or 
\item the hyperbolic space $\mathbb{H}^n$ or
\item $\mathbb{R}\times_{\varphi}\Sigma_c,$
\end{enumerate} where $\Sigma_c$ is a regular level set of the potential function, which is an Einstein mani\-fold with $Ric_{\Sigma_c}=\lambda g_{\Sigma_c},$ and the warping function $\varphi$ is a solution of the ODE
\begin{equation}
\label{ODEgeral1}
     \varphi\Big(\frac{R}{n-1}\varphi + 2\varphi''\Big) + (n-2)(\varphi')^2 = \lambda.
\end{equation}
\end{theorem}

\begin{proof}
By analyticity \cite[Proposition 2.1]{corv2013def}, it suffices to prove the result on $\Omega = \{p \in M:\,|\nabla f| \neq 0\},$ which exists by Sard's Theorem and the fact that $f$ is nontrivial. In the first part of the proof, we shall adapt some ideas outlined in \cite{catino,CaoYu}. First, pick a regular value $c$ of the potential function $f$ and let $\Sigma_c$ be the level surface, that is, $\Sigma_c =f^{-1}(c)$. Let $I \in \mathbb{R}$ be an open interval containing $c$ such that $f$ has no critical points in the open neighborhood $U_I =f^{-1}(I)$ of $\Sigma_c$. Since $(M^n,\,g)$ satisfies $B(\nabla f,\nabla f)=0$ and $f$ is proper, Lemma \ref{lem5} guarantees that $(M^n,\,g)$ has vanishing $T$ tensor\footnote{Obviously, the result also holds if $(M^n,\,g)$ is Bach-flat.}. Therefore, Proposition \ref{Prop2} guarantees that $U_I$ admits a local warped product structure 
    \begin{equation}\label{eqwarp}
        g= dr^2 + \varphi^2(r)g_{\Sigma_c},
    \end{equation}
    where
    \begin{equation*}
        r(x) = \int_{f=c}^{f}\frac{df}{|\nabla f|}
    \end{equation*}
  is the signed distance from the regular hypersurface $\Sigma_c$, which is a totally umbilical Einstein hypersurface. In particular, since $M^n$ is simply connected, then $\Sigma_c$ is a two-sided embedded hypersurface.

In the second part of the proof, we aim to extend the local warped product structure to a global one. To this end, following~\cite{CaoYu} and \cite{lopez2014note}, we divide the argument into a sequence of claims.\\
   
{\bf Claim 1.} The interval in which the warped structure is defined can be extended as long as the warping function does not vanish.\\

Let $I_{\max}\supseteq I$ be the maximal interval for which the warped structure (\ref{eqwarp}) holds. Observe that if $|\nabla f| $ never vanishes, the claim follows immediately. Otherwise, assume that $|\nabla f|(\bar{r})=0$ for some $\bar{r},$ so that either $(\bar{r}-\varepsilon,\bar{r})\subset I_{\max}$ or $(\bar{r},\bar{r}+\varepsilon)\subset I_{\max}$, that is, $\bar{r}$ is an endpoint of the interval $I_{\max}$. Thus, since $\Omega$ is dense, by continuity and smoothness, $\Sigma_{\bar{c}} =f^{-1}(\bar{c}) =r^{-1}(\bar{r})$ is a totally umbilical submanifold and the warped structure can be extended to $(\bar{r}-\varepsilon,\bar{r}+\varepsilon)\subset I_{\max}$ for $\varepsilon >0$ sufficient small.

We now have three cases: either $\varphi$ never vanishes, in which case $I_{\max} = \mathbb{R}$; or it vanishes at exactly one point, implying that $I_{\max} = (a,\infty)$ with $\lim_{r\rightarrow a}\varphi(r) = 0 $; or it vanishes at precisely two points, so that $I_{\max}=(a,b)$ and $\lim_{r\rightarrow a}\varphi(r) = \lim_{r\rightarrow b}\varphi(r)=0$. Therefore, by a suitable shifting in $r,$ one sees that $(M^n,\,g)$ is isometric to one of the following possibilities: 

    \begin{enumerate}
       \item[(i)] $\mathbb{R}\times_{\varphi}\Sigma_c$; or
        \item[(ii)]$[0,\infty)\times_{\varphi}\Sigma_c$; or
        \item[(iii)] $[0,a]\times_{\varphi}\Sigma_c$. 
    \end{enumerate}
    
    \vspace{0.20cm}
    
    {\bf Claim 2.} If $(M^n,\,g)$ is isometric to either $[0,\infty)\times_{\varphi}\Sigma_c$ or $[0,a]\times_{\varphi}\Sigma_c$, then $\Sigma_c$ is isometric to $\mathbb{S}^{n-1}$.\\

We address here the first case; the second is analogous. Suppose that $(M^n,\,g)$ is isometric to $[0,\infty)\times_{\varphi}\Sigma_c,$ where the warped product metric extends smoothly to $\{r=0\}$. Since $\varphi(0)=0$, smoothness of the metric requires that the slice $\{0\}\times\Sigma_c$ collapses to a single point, denoted by $p \in M$. Consequently, the exponential map at $p$ ensures that $\Sigma_c$ is diffeomorphic to $\mathbb{S}^{n-1}.$ Moreover, in a neighborhood of $p,$ the metric can be expressed in polar coordinates as
    \begin{equation}
    \label{eqpolarA}
        g = dr^2 + r^2h_r(\theta),
    \end{equation}
   where $h_r$ is a smooth one-parameter family of metrics on $\mathbb{S}^{n-1}$, and $h_0$ is the standard round metric. So, taking into account (\ref{eqwarp}), we now need to show that $h_{0}=g_{\Sigma_{c}}.$ Indeed, following \cite[Claim 4]{catino} (see also \cite[Lemma 9.114]{besse2007einstein}), the Taylor expansion of $\varphi$ has the form 
   
      \begin{equation*}
       \varphi(r) = r + ar^3 + O(r^5),
   \end{equation*} so that
      \begin{equation}
       \varphi^2(r) = r^2 + 2ar^4 + O(r^6).
   \end{equation} In conjunction with (\ref{eqpolarA}) and (\ref{eqwarp}) yields

   \begin{equation*}
       \varphi^2(r)g_{\Sigma_c} = r^2h_r(\theta).
   \end{equation*}
   Consequently,
   \begin{equation*}
       r^2g_{\Sigma_c} + 2ar^2g_{\Sigma_c} + O(r^6)g_{\Sigma_c} = r^2h_r(\theta).
   \end{equation*} Dividing the above equation by $r^2$ for $r>0$ and letting $r\to 0,$ one sees that
      \begin{equation*}
       h_0 = \lim_{r\rightarrow0}h_r = g_{\Sigma_c}.
   \end{equation*} Therefore, the metric on $\Sigma_c$ must be the round metric, and hence $\Sigma_c$ is isometric to $\mathbb{S}^{n-1},$ which completes the proof of Claim 2.\\

To proceed, observe that if a manifold $(M^n,\,g)$ is expressed as a warped product of an interval and an Einstein manifold, then the following ODE holds
    \begin{equation}\label{eqODEgeral2}
        R = \frac{n-1}{\varphi^2}(\lambda - 2\varphi\varphi'') - (n-1)(n-2)\Big(\frac{\varphi'}{\varphi}\Big)^2,
    \end{equation} where $\lambda$ is the Einstein constant and $\varphi$ denotes the warping function. We now distinguish three cases, which will be analyzed separately.\\

\textbf{\underline{Case 1:}} $(M^n,\,g)$ is isometric to $[0,a]\times_{\varphi}\mathbb{S}^{n-1}$.\\
     
    In this case, we have $\lambda = n-2$ and (\ref{eqODEgeral2}) becomes
    \begin{equation}\label{ODEsimple}
        \frac{R}{n-1}\varphi^2 = (n-2)(1-(\varphi')^2)-2\varphi\varphi'',
    \end{equation}
    with $\varphi(0)=0$. In particular, for the metric to extend smoothly, we have $\varphi'(0)=1$. We then solve this initial value problem, which depends only on the sign of the scalar curvature $R$ as follows.

    \begin{enumerate}
        \item[(i)] If $R>0$, then
             $ \varphi(r) = \sqrt{\frac{n(n-1)}{R}}\sin\Bigg(\sqrt{\frac{R}{n(n-1)}}r \Bigg)$
             and $[0,a] = \Big[0,\sqrt{\frac{n(n-1)}{R}}\pi\Big]$. Hence, $(M^n,\,g)$ is isometric to the standard sphere $\mathbb{S}^n$.
             \item[(ii)] If $R=0$, then $\varphi(r) = r.$ As observed in the proof of Claim~1, for such an isometry to occur, the warping function would have to vanish at exactly two points; therefore, this case cannot arise.
             \item[(iii)] If $R<0$, then $ \varphi(r) = \sqrt{-\frac{n(n-1)}{R}}\sinh\Bigg(\sqrt{-\frac{R}{n(n-1)}}r \Bigg).$ As in case (ii), this situation cannot occur since $\varphi$ vanishes at only one point.\\
    \end{enumerate}

\textbf{\underline{Case 2:}} $(M^n,\,g)$ is isometric to $[0,\infty)\times_{\varphi}\mathbb{S}^{n-1}$.\\

Similar to the previous case, the initial data are given by
 $\varphi(0)=0$ and $\varphi'(0) =  1,$ and the corresponding ODE coincides with (\ref{ODEsimple}). Therefore, the possible solutions $\varphi(r)$ are the same as in Case 1.
\begin{enumerate}
    \item[(i)] If $R>0$, one observes that $\varphi$ vanishes at two points. However, from the proof of Claim~1, we know that in the isometric case the warping function must vanish at exactly one point. Therefore, the case $R>0$ cannot occur.
     \item[(ii)] If $R=0$, one deduces that $\varphi(r) = r.$ Therefore, $(M^n,\,g) $ is isometric to the Euclidean space $\mathbb{R}^n.$
         \item[(iii)] If $R<0$, then $ \varphi(r) = \sqrt{-\frac{n(n-1)}{R}}\sinh\Bigg(\sqrt{-\frac{R}{n(n-1)}}r \Bigg).$ Thus, $(M^n,\,g)$ is isometric to the hyperbolic space $\mathbb{H}^n$.
\end{enumerate}

\textbf{\underline{Case 3}:} $(M^n,\,g)$ is isometric to $\mathbb{R}\times_{\varphi}\Sigma_c$.\\

In this case, it follows from (\ref{eqODEgeral2}) that
\begin{equation}
\label{eqJNK1}
     \varphi\Big(\frac{R}{n-1}\varphi + 2\varphi''\Big) + (n-2)(\varphi')^2 = \lambda.
\end{equation}

This analysis covers all possible cases, and the theorem is thus proved.
\end{proof}


\begin{remark}
We note that $\varphi$ satisfying $$(\varphi')^2=\frac{\lambda}{n-2}-\frac{R}{n-1}\varphi^2$$ is a solution of (\ref{eqJNK1}). In particular, differentiating this identity gives $\varphi'' =-\frac{R}{n-1}\varphi,$ which provides a direct solution of (\ref{eqJNK1}). However, since no initial data are available in Case~3, we cannot assert the uniqueness of such a solution. Nevertheless, we have the following:

\begin{itemize}
\item If $R\geq 0$ and $\lambda\leq 0,$ then (\ref{eqJNK1}) implies $\varphi''\leq 0.$ Because $\varphi$ is a positive and weakly concave function on $\Bbb{R},$ it must be constant. Substituting this into (\ref{eqJNK1}) then forces $R=\lambda=0.$ Consequently, by \cite[Theorem 2.2]{miao2011einstein}, $M^n$ is isometric to the Euclidean space $\mathbb{R}^n.$

\end{itemize} 
\end{remark}


We now turn our attention to the lower dimensional cases.

\begin{proof}[{\bf Proof of Corollary \ref{Cor1}}]

First, notice that items (1)–(3) of Corollary~\ref{Cor1} follow immediately from Theorem~\ref{ThmB}, so we are left to analyze the last case. Since the function $\varphi$ is not determined by an initial value problem, we are unable to explicitly solve it; however, we can still make progress in understanding the structure of $\Sigma_c$.

Observe that, under the assumption that $M^4$ is simply connected, the level set $\Sigma_c$  is necessarily simply connected as well. Therefore, $\Sigma_c$ is a compact, simply connected, $3$-dimensional Einstein manifold. Thus, it follows from the Killing--Hopf theorem that $\Sigma_c$ is isometric to $\mathbb{S}^3.$ Finally, by substituting $n=4$ and $\lambda = 2$ into the general ODE~(\ref{ODEgeral1}), we obtain
\begin{equation*}
    \varphi\left(\varphi'' + \frac{R}{6}\varphi\right) + (\varphi')^2 =1,
\end{equation*}
which characterizes the warping function $\varphi,$ and completes the proof.
\end{proof}

We now present the proof of Theorem \ref{ThmC}.

\begin{proof}[{\bf Proof of Theorem \ref{ThmC}}] In the first part of the proof, we follow \cite[Theorem 1.2]{catino}. In dimension $3,$ by a standard argument involving the definitions of the Bach and Cotton tensors, we infer

\begin{equation}
\label{eqkjnm560}
 \nabla_i B_{ij}\nabla_jf = -R_{ik}C_{jki}\nabla_{j}f=-\frac{1}{2}\Big(C_{jki}R_{ik}\nabla_jf + C_{kji}R_{ij}\nabla_kf\Big);
 \end{equation} for more details, see \cite[Eq. (4.2)]{catino} and  \cite[Eq. (3.3)]{barros2015bach}.

Plugging (\ref{eqT}) and the fact that $C$ is trace-free into (\ref{eqkjnm560}), one obtains that
\begin{equation*}
    div\,B(\nabla f) = \frac{1}{2}C_{kji}\left( R_{ik}\nabla_{j}f-R_{ij}\nabla_{k}f\right)=\frac{1}{4}C_{kji}T_{kji}.
\end{equation*}
Applying Lemma \ref{lem2} together with the fact that $W=0,$ we deduce
\begin{equation}
    div\,B(\nabla f) = \frac{f}{4}|C|^2.
\end{equation} Thus, assuming that $ div\,B(\nabla f)=0 $ -- in particular, when $M^3$ is Bach-flat -- we conclude from  Lemma \ref{lem2} that $T\equiv 0.$ It then follows from Proposition~\ref{Prop2} that $(M^3,\,g)$ is locally a warped product. Consequently, repeating the arguments from the second part of the proof of Theorem~\ref{ThmB}, we deduce that $M^3$ is isometric to one of the following spaces: the standard sphere $\mathbb{S}^3,$ the Euclidean space $\mathbb{R}^3,$ the hyperbolic space $\mathbb{H}^3,$ or $\mathbb{R}\times_{\varphi}\Sigma_c,$
where $\Sigma_c$ is a regular level set of the potential function and the warping function $\varphi$ satisfies
\begin{equation}
\label{aaaaa}
     \varphi\Big(\frac{R}{2}\varphi + 2\varphi''\Big) + (\varphi')^2 = \lambda.
\end{equation}

In the last case, by \cite[Proposition 3.1 (i)]{miao2011einstein}, we have that $\Sigma_{c}$ is Einstein mani\-fold, i.e., $Ric_{\Sigma_c}=\lambda g_{\Sigma_c}$ with $\lambda$ constant. Therefore, $\Sigma_c$ has constant sectional curvature (see \cite[Proposition 1.120]{besse2007einstein}). By using the Killing--Hopf theorem and the fact that $f$ is proper, one concludes that the level set $\Sigma_c$ is isometric to $\Bbb{S}^2.$ This completes the proof of the theorem.
\end{proof}

\vspace{0.20cm}

\noindent{\bf Conflict of Interest:} There is no conflict of interest to disclose.
	
	\

\noindent{\bf Data Availability:} Not applicable.

\end{document}